\documentclass[11pt,reqno]{amsart} 
\IfFileExists{srcltx.sty}{\usepackage[active]{srcltx}}{}

\usepackage{amscd}
\usepackage{amsfonts,amssymb,latexsym} 
\usepackage[all]{xy} 
\xyoption{arc}
\CompileMatrices

\newcommand{\udots}{\mathinner{\mskip1mu\raise1pt\vbox{\kern7pt\hbox{.}}
\mskip2mu\raise4pt\hbox{.}\mskip2mu\raise7pt\hbox{.}\mskip1mu}}

\newcommand{\SO}{{\mathcal{O}}}

\newcommand{\SQ}{{\mathcal{Q}}}

\newcommand{\PP}{\mathbb{P}}

\newcommand{\Pic}{\operatorname{Pic}}

\newcommand{\Sym}{\operatorname{Sym}}
\newcommand{\Aut}{\operatorname{Aut}}

\newtheorem{proposition}{Proposition}[section]
\newtheorem{theorem}[proposition]{Theorem}

\newtheorem{lemma}[proposition]{Lemma}

\numberwithin{equation}{section}

\title[Automorphisms of a symmetric product of a curve]{Automorphisms of a
symmetric product of a curve}

\author[I. Biswas]{Indranil Biswas}
\address{School of Mathematics, Tata Institute of Fundamental
Research, Homi Bhabha Road, Bombay 400005, India}
\email{indranil@math.tifr.res.in}

\author[T. L. G\'omez]{Tom\'as L. G\'omez}
\address{Instituto de Ciencias Matem\'aticas (CSIC-UAM-UC3M-UCM),
C/ Nicolas Cabrera 15, 28049 Madrid, Spain}
\email{tomas.gomez@icmat.es}

\thanks{The authors want to thank the support by MINECO (ICMAT Severo
Ochoa project SEV-2011-0087 and grant MTM2013-42135-P) and the
Spanish Ministry of Education for the grant ``Salvador de
Madariaga'' PRX14/00508. We acknowledge the support of the grant 612534 MODULI
within the 7th European Union Framework Programme.} 



\begin{document}

\begin{abstract}
We show that all the automorphisms of the symmetric product 
$\Sym^d (X)$, $d\,>\,2g-2$, of a smooth projective curve $X$
of genus $g\,>\,2$ are induced by automorphisms of $X$.
\end{abstract}

\maketitle

\section{Introduction}

Let $X$ be a smooth projective curve of genus $g$, with $g\,>\,2$,
over an algebraically closed field. Take
any integer $d\,>\,2g-2$. Let $\Sym^d (X)$ be the $d$-fold symmetric product
of $X$. Our aim here is to study the group $\Aut(\Sym^d (X))$ of automorphisms
of the algebraic variety $\Sym^d (X)$. 
An automorphism $f$ of the algebraic curve $X$ produces an algebraic automorphism
$\rho(f)$ of $\Sym^d (X)$ that sends any $\{x_1\, , \cdots\, , x_d\}\,\in\, \Sym^d (X)$ to
$\{f(x_1)\, , \cdots\, , f(x_d)\}$. This map
$$
\rho,:\, \Aut(X) \,\longrightarrow\, \Aut(\Sym^d (X))\, ,\ \
f\,\longmapsto\,\rho(f)
$$
is a homomorphism of groups.

\begin{theorem}\label{thmi}
The natural homomorphism 
$$
\rho\,:\, \Aut(X) \,\longrightarrow\, \Aut(\Sym^d X)
$$
is an isomorphism.
\end{theorem}

The idea of the proof is as follows.
The homomorphism $\rho$ is clearly injective, so we have to show that
it is also surjective.
The Albanese variety of $\Sym^d (X)$ is the Jacobian $J(X)$ of $X$. So an
automorphism of $\Sym^d (X)$ induces an automorphism of $J(X)$. Using
results of Fakhruddin and Collino--Ran, we show that the induced
automorphisms of $J(X)$ respects the theta divisor up to
translation. Invoking the strong form of the Torelli theorem for the
Jacobian, it follows that such automorphisms are generated by automorphisms of the
curve $X$, translations of $J(X)$, and the inversion of $J(X)$ that sends each line bundle
to its dual. Using a result of Kempf we show that if an automorphism $\alpha$
of $J(X)$ lifts to $\Sym^d (X)$, then $\alpha$ is induced by an
automorphism of $X$, and this finishes the proof.

It should be clarified that we need a slight generalisation of the result of Kempf
\cite{K}; this is proved in Section
\ref{sec:preliminaries}. The proof of Theorem \ref{thmi} is in Section
\ref{sec:proof}.

\medskip
\noindent{\bf Acknowledgements.}\, We thank N. Fakhruddin for helpful
discussions. The second author also wants to thank Tata Institute for
Fundamental Research and the University of Warwick, where part of this
 work was done. 

\section{Some properties of the Picard bundle}
\label{sec:preliminaries}

A \emph{branding} of a Picard variety $P^d\,=\,\Pic^d(X)$ is a Poincar\'e line bundle
$\SQ$ on $X\times P^d$ \cite[p. 245]{K}. Two brandings differ by the pullback of
a line bundle on $P^d$.

The degree of a line bundle $\xi$ over a smooth projective variety $Z$ is the class
of the first Chern class $c_1(\xi)$ in the N\'eron-Severi group
$\text{NS}(Z)$, so the line bundles of
degree zero on $Z$ are classified by the Jacobian $J(Z)$.
A \emph{normalised branding} is a branding such that $\SQ\vert_{\{x\}\times P^d}$ has
degree zero for one point $x\,\in\, X$ (equivalently, for all points of $X$).
Two normalised brandings differ by the pullback of a degree zero line
bundle on $P^d$.

The natural projection of $X\times P^d$ to $P^d$ will be denoted by $\pi_{P^d}$.
A normalised branding $\SQ$ induces an embedding 
\begin{equation}\label{e1}
I_{\SQ}\,:\,X\,\longrightarrow\, J(P^d)\,=:\,J
\end{equation}
 that sends any $x\,\in\, X$ to the
point of $J$ corresponding to the line bundle $\SQ\vert_{\{x\}\times P^d}$.
If $\SQ'\,=\,\SQ\otimes \pi^*_{P^d} L_j$, where
$L_j$ is the line bundle corresponding to a point $j\,\in\, J(P^d)\,=\,J$, then
we have $I_{\SQ'}\,=\,I_{\SQ}+j$.

Assume that $d\,>\,2g-2$. A \emph{Picard bundle} $W(\SQ)$ on $P^d$ is the vector bundle 
$\pi_{P^d*} \SQ$, where $\SQ$ is a normalised branding. From the projection formula it
follows that two Picard bundles differ by tensoring with a degree zero line 
bundle on $P^d$.

There is a version of the following proposition for $d\,<\,0$ in 
\cite[Corollary 4.4]{K} (for negative degree, the Picard bundle is
defined using the first direct image).

\begin{proposition}
\label{K4.4}
Let $d\,>\, 2g-2$.
\begin{enumerate}
\item $H^1(P^d,\,W(\SQ))$ is non-zero (in fact, it is one-dimensional if it is
non-zero) if and only if $0\,\in\, I_\SQ(X)$.\label{part1}

\item Let $L_j$ be the line bundle on $P^d$ corresponding to a point $j\,\in\,
J$. Then $H^1(P^d,\, L_j\otimes W(\SQ))$ is non-zero (in fact, it is
one-dimensional if it is non-zero) if and only if $-j\,\in\, I_{\SQ}(X)$.
\label{part2}
\end{enumerate}
\end{proposition}

\begin{proof}
Part \eqref{part1}. If $0\,\notin\, I_\SQ(X)$, then $H^1(P^d,\,W(\SQ))\,=\,
0$ by \cite[p. 252, Theorem 4.3(c)]{K}. 
Fix a line bundle $M$ on $X$ of degree one, and consider the associated
Abel-Jacobi map
$$
X\,\longrightarrow\, J(X),\ \ x\,\longmapsto\, M\otimes {\mathcal O}_X(-x)\, .
$$
Let
$N_{X/J(X)}$ be the normal bundle of the image of $X$ under this Abel-Jacobi map. 
If $0\,\in\, I_\SQ(X)$, then using \cite[p. 252, Theorem 4.3(d)]{K} it follows that
$H^1(P^d,\,W(\SQ))$ is isomorphic to the space of sections of the skyscraper sheaf on $X$
$$
K_X^{-1}\otimes \wedge^0 N_{X/J}\otimes M^d|_{I_\SQ^{-1}(0)}\, ,
$$
where $I_{\SQ}$ is constructed in \eqref{e1}. But the space of sections of this
skyscraper sheaf is clearly one-dimensional, because $I_\SQ^{-1}(0)$ consists of
one point of $X$.

Part \eqref{part2} follows from part \eqref{part1} because 
$L_j\otimes W(\SQ)\,=\, W(\SQ\otimes\pi^*_{P^d} L_j)$, and
$I_{W(\SQ\otimes\pi^*_{P^d} L_j)}\,= \,I_{W(\SQ)}+j$.
\end{proof}

\begin{proposition}
\label{prop91}
Assume $g(X)\,>\,1$ and $d\,>\,2g-2$. Let $j$ be a point of
${\rm Pic}^0(X)$, and let $T_j\,:\, P^d\,\longrightarrow\, P^d$ be
the translation by $j$. Let $M$ be a degree zero line
bundle on $P^d$. If
$$
T^*_j (M\otimes W(\SQ)) \,\cong\, W(\SQ)\, ,
$$
then $j\,=\,0$ and $M\,=\,\SO_{P^d}$.

Let $i\,:\,P^d\,\longrightarrow\, P^d$ be the inversion given by $z\,\longmapsto\,
-z+2z_0$, where $z_0$ is a fixed point in $P^d$. If
$$
i^*T^*_j (M\otimes W(\SQ)) \,\cong\, W(\SQ)\, ,
$$
then $X$ is a hyperelliptic curve.
\end{proposition}

\begin{proof}
The first part is \cite[Proposition 9.1]{K} except that there it
is assumed that $d\,<\,0$; the proof of Proposition 9.1 uses 
\cite[Corollary 4.4]{K} which requires this hypothesis. However, the case
$d\,>\,2g-2$ can be proved similarly; for the convenience of the reader we
give the details.

Let $y\,\in\, J$ be the point corresponding to the line bundle $M$.
The line bundle on $P^d$ corresponding to any $t\,\in\, J$ will be denoted
by $L_t$. In particular, $M\,=\,L_y$. For every $t\,\in\, J$,
using the hypothesis, we have
\begin{equation}
\label{eq:1}
T^*_j(L_{t+y}\otimes W(\SQ))\,=\,T^*_jL_t \otimes T^*_j(M\otimes W(\SQ))
\,=\,L_t\otimes W(\SQ)\, ;
\end{equation}
note that the fact that a degree zero line
bundle on an Abelian variety is translation invariant is used
above. Combining \eqref{eq:1} and the fact that $T_j$ is an isomorphism, we have
$$
H^1(L_t\otimes W(\SQ)) \,\cong\, 
H^1(T^*(L_{t+y}\otimes W(\SQ))) \,\cong\, 
H^1(L_{t+y}\otimes W(\SQ))\, .
$$
Using \cite[Corollary 4.4]{K} it follows that
$t\,\in\, -I_{\SQ}(X)$ if and only if $t+y\,\in\, -I_{\SQ}(X)$.
Hence $I_{\SQ}(X)\,=\,y+I_{\SQ}(X)$. If $g(X)\,>\,1$, this implies
that $y\,=\,0$. Therefore, we have $W(\SQ)\,=\,T^*_j(W(\SQ))$. Using the fact
that $c_1(W(\SQ))\,=\,\theta$, a theta divisor, it follows that
$\theta$ is rationally equivalent to the translate $\theta-j$,
hence $j\,=\,0$.

The proof of the second part is similar. We have
\begin{equation}
\label{eq:2}
i^*T^*_j(L_{y-t}\otimes W(\SQ))\,=\,i^*T^*_jL_{-t} \otimes
i^*T^*_j(M\otimes W(\SQ))
\,=\,i^*L_{-t}\otimes W(\SQ)\,=\,L_{t}\otimes W(\SQ)\, ;
\end{equation}
the fact that $i^*L_{-t}\,=\,L_{t}$ is used above. Consequently,
$$
H^1(L_t\otimes W(\SQ)) \,\cong\, 
H^1(i^*T^*(L_{y-t}\otimes W(\SQ))) \,\cong\, 
H^1(L_{y-t}\otimes W(\SQ))\, ,
$$
and using \cite[Corollary 4.4]{K} it follows that
$t\,\in\, -I_{\SQ}(X)$ if and only if $y-t\,\in\, -I_{\SQ}(X)$.
Hence $I_{\SQ}(X)\,=\,-I_{\SQ}(X)-y$. Let
$$
f\,:\,X\,\longrightarrow\, X
$$
be the morphism
uniquely determined by the condition $$I_{\SQ}(x)\,=\,-I_{\SQ}(f(x))-y\, .$$
We note that $f$ is well defined because $-I_{\SQ}$ and $y+I_{\SQ}$ are 
two embeddings of $X$ in $J$ with the same image, 
so they differ by an automorphism of $X$
which is $f$. In other words, if we identify $X$ with its image
under $I_{\SQ}$, then $f$ is induced from the automorphism 
$T_{-y}\circ i$ of $J$. This automorphism $T_{-y}\circ i$ is clearly an involution.
Let $\omega\in H^0(C,\, \Omega_C)$ be a holomorphic 1-form on $X$. Then 
$f^*\omega\,=\,-\omega$, because of the isomorphism 
$H^0(X,\, \Omega_X)\,=\, H^0(J, \Omega_J)$ induced by $I_{\SQ}$,
and the fact that $i^*$ acts as multiplication by $-1$ on the 1-forms
on $J$. It now
follows by Lemma \ref{invol} that $f$ is a hyperelliptic involution.
\end{proof}

\begin{lemma}
\label{invol}
Let $g\,>\,1$. 
Let $f\,:\,X\,\longrightarrow\, X$ is an involution satisfying the condition
that $f^*\omega=-\omega$ for every $1$-form $\omega$. Then $X$ is hyperelliptic with
$f$ being the hyperelliptic involution.
\end{lemma}

\begin{proof}
Consider the canonical morphism 
$$
F\,:\,X \,\longrightarrow\, \PP(H^0(X,\,K_X))
$$
that sends any $x\,\in \,X$ to the hyperplane $H^0(K_X(-x))$
in $H^0(K_X)$. By definition,
$$
H^0(\Omega_X(-x))=\{\omega\,\in\, H^0(\Omega_X)\,\mid\,\omega(x)=0\}\, ,
$$
but the hypothesis implies that $\omega(x)\,=\,0$ if and only if 
$\omega(f(x))\,=\,f^*(\omega)(x)\,=\,0$. Therefore, we have
$$
H^0(\Omega_X(-x))\,=\, H^0(\Omega_X(-f(x)))\, ,
$$
and it follows that $F(x)\,=\,F(f(x))$, implying that the canonical morphism is not
an embedding; note that $f$ is not the identity because there are nonzero
holomorphic 1-forms. Therefore, $X$ is hyperelliptic, and $f$ is the
hyperelliptic involution.
\end{proof}

\section{Proof of Theorem \ref{thmi}}
\label{sec:proof}

Using the morphism $X\, \longrightarrow\, \text{Sym}^d(X)$, $y\, \longmapsto\,
dy$, it follows that the homomorphism $\rho$ in Theorem \ref{thmi} is injective.

Fix a point $x\,\in\, X$. Let $\mathcal{L}$ be the normalised Poincar\'e
line bundle on $X\times J(X)$, i.e., it is trivial when restricted to the
slice $\{x\}\times J(X)$. Let
$$
E\,:=\,q_*(\mathcal{L}\otimes p^*\SO_X(dx))
$$
be the Picard bundle,
where $p$ and $q$ are the projections from $X\times J(X)$ to $X$ and $J(X)$
respectively. Since $d\,>\, 2g-2$, it follows that $E$ is a vector bundle
of rank $d-g+1$.

We will identify $\Sym^d(X)$ with the projective bundle $P(E)\,=\, \PP(E^\vee)$.

Let $\theta$ be the theta divisor of $J(X)$; in particular, we have 
$\theta^g\,=\, g!$. The Chern
classes of $E$ are given by $c_i(E)\,=\,\theta^i/i!$ \cite{ACGH}.

The Albanese variety of $P(E)$ is the Jacobian $J(X)$,
and the Albanese map sends an effective divisor $\sum_{\ell=1}^d P_\ell$ of
degree $d$ to the degree zero line bundle $\SO_X((\sum_{\ell=1}^d P_\ell) - dx)$.
Given an automorphism
$$\varphi\,:\, P(E)\,\longrightarrow\, P(E)\, ,$$ the universal
property of the Albanese variety yields a commutative diagram
\begin{equation}
\label{eq:albanese}
\xymatrix{
{P(E)} \ar[r]_{\cong}^{\varphi} \ar[d] & {P(E)} \ar[d]\\
{J(X)} \ar[r]_{\cong}^{\alpha} & {J(X)}
}
\end{equation}
and this produces an automorphism of projective bundles
$$
\xymatrix{
{P(E)} \ar[rr]^{\psi}_{\cong} \ar[rd] &&
{P(\alpha^* E)}\ar[ld] \\
& {J(X)} &
}
$$
Therefore, there is a line bundle $L$ on $J(X)$ such that there is an isomorphism
\begin{equation}\label{eq:fl}
\alpha^* E \,\cong\, E\otimes L\, .
\end{equation}
There is a commutative diagram of groups
\begin{equation}\label{eq:diag}
\xymatrix{
\Aut(P(E)) \ar[r]^\lambda & \Aut(J(X))\\
\Aut(X) \ar[u]^\rho \ar[ur]_{\mu}
}
\end{equation}
where $\lambda$ is constructed as above using the universal property of the Albanese
variety given in \eqref{eq:albanese}, and $\rho$ is the homomorphism in Theorem
\ref{thmi}. To construct $\mu$, note that the commutativity of the diagram
\eqref{eq:albanese} implies that $\mu(f)$, $f\, \in\, \Aut(X)$, 
has to send $\SO_X((\sum_{\ell=1}^d P_\ell)-dx)$
to $\SO_X((\sum_{\ell=1}^d f(P_\ell))-dx)$. A short calculation yields
\begin{equation}\label{eq:muf}
\mu(f)\,=\,(f^{-1})^*\circ T_{dx-df^{-1}(x)}\, ,
\end{equation}
where $T_a$, $a\,\in\, J(X)$, is translation on $J(X)$ by $a$. 

Let $\theta'\,=\,c_1(\alpha^*E)\,=\,\theta+L$. Then
$$
c_i(\alpha^* E) \,=\, \alpha^* c_i(E)\,=\, \frac{\alpha^* \theta^i}{i!}\,
= \,\frac{ \theta'{}^i}{i!}\, ,
$$
and $\theta'{}^g\,=\,\alpha^* \theta^g\,=\,g!$.
Now we apply \cite[Lemma 1]{F}; here the condition $g\,>\,2$ is used. 
We obtain $\theta^i\,=\,\theta'{}^i$ for all $i\,>\,1$.

We identify $X$ with the image in $J(X)$ of the Abel-Jacobi map. 
In particular $X$ is numerically equivalent to $\theta^{g-1}/(g-1)!$.
We calculate the intersection (note that the condition $g>2$ is again used, because
we need $g-1>1$)
$$
\theta' X = \theta' \frac{\theta^{g-1}}{(g-1)!} = 
\theta' \frac{\theta'{}^{g-1}}{(g-1)!} = g\, .
$$
Invoking a characterisation of a Jacobian variety due to Collino and Ran,
\cite{C}, \cite{R}, it follows that $(J(X)\, ,\theta'\, , X)$ is a Jacobian triple,
i.e., $\theta'$ is a theta divisor of the Jacobian variety $J(X)$ up
to translation. This means that $\theta$ and $\theta'$ differ by
translation, in other words, the class of $c_1(L)$ in the N\'eron-Severi group
$\text{NS}(J(X))$ is zero. Consequently, $\alpha$ is an isomorphism of polarised Abelian
varieties, i.e., it sends $\theta'$ to a translate of it.

The strong form of the classical Torelli 
theorem (\cite[Th\'eor\`eme 1 and 2 of Appendix]{LS}) tells us that
such an automorphism $\alpha$ is of the form
$$
\alpha\,=\,F \circ \sigma \circ T_a\, ,\ \ \sigma\,\in\, \{1\, ,\iota\}\, ,
$$
where $F\,=\,(f^{-1})^*$ for an automorphism $f$ of $X$, while $T_a$ is
translation by an element $a\,\in\, J(X)$ 
and $\iota$ sends each element of $J(X)$ to its inverse.
If $X$ is hyperelliptic, then $\iota$ is induced by
the hyperelliptic involution, so we may assume that $\sigma$ is the
identity map of $X$ when $X$ is hyperelliptic.

Let $f$ be an automorphism of $X$ with $F\,=\,(f^{-1})^*$ being the induced
isomorphism on $J(X)$. Using the definition of $E$, it is easy to check that 
$$
F^* E \,\cong \,T^*_{dx'-dx} E\, ,
$$
where $x'\,=\,f^{-1}(x)$. 

We claim that $\alpha\,=\,F\circ T_a$.

To prove this, assume that $\alpha\,\not=\,F\circ T_a$. Then
$X$ is not hyperelliptic, and $\alpha\,=\,F\circ \iota \circ T_a$. Hence
$$
\alpha^* E\,=\,T_a^* \iota^* F^* E \,=\, T_a^* \iota^* T_{dx'-dx}^* E \, ,
$$
and using \eqref{eq:fl},
$$
E \,\cong\, \iota^* T^*_{dx'-dx-a}(E\otimes L)\, .
$$
Now from Proposition \ref{prop91} it follows that $X$ is hyperelliptic, and
we arrive at a contradiction. This proves the claim.

Summing up, we can assume that $\alpha\,=\,F\circ T_a$. 
Using \eqref{eq:fl}, 
$$
 E \,\cong \, T^*_{dx-dx'-a}(E\otimes L)
$$
{}From Proposition \ref{prop91} it follows that $L$ is the trivial line
bundle, and $a\,=\,dx-dx'$. Therefore,
$$\alpha\,=\,(f^{-1})^* \circ T_{dx-df^{-1}(x)}$$ for some automorphism $f$ of 
$X$, and hence, by \eqref{eq:muf}, 
\begin{equation}\label{eq:inj}
{\rm Image}(\lambda)\, \subset\, {\rm Image}(\mu)\, .
\end{equation}

We will now show that the morphism $\lambda$ is injective.

Suppose $\alpha\,=\,\lambda(\varphi)\,=\,{\rm Id}_{J(X)}$. 
Using \eqref{eq:fl}, the morphism $\varphi$ is induced by
an isomorphism between $E$ and
$E\otimes L$. We have just seen that $L$ is trivial, the
morphism $\varphi$ is induced by an automorphism of $E$,
and this automorphism has to be
multiplication by a nonzero scalar, because $E$ is stable with respect
to the polarisation given by the theta divisor (cf. \cite{EL}).
Therefore, the morphism $\varphi$ is the identity.
This proves that the morphism $\lambda$ is injective.

The homomorphism $\mu$ is also injective, since it is a composition of
a translation and the pullback induced by an automorphism of $X$.

Combining these it follows that the morphism $\rho$ is also injective (this can also
be checked directly), and hence
all the homomorphisms in the diagram \eqref{eq:diag} are
injective. This, combined with \eqref{eq:inj}, 
shows that $\rho$ is an isomorphism.

\end{document}